\documentclass[reqno,12pt,a4paper,oneside]{amsart}

\usepackage{amsmath}
\usepackage{array}
\usepackage{mathrsfs}
\usepackage{amssymb}
\usepackage{amsfonts}
\usepackage{latexsym}
\usepackage{amsthm}
\usepackage{amsmath}
\usepackage[arrow, matrix, curve]{xy}
\usepackage{graphicx}	
\usepackage{tikz}
\usetikzlibrary{matrix,arrows}
\input xy

\theoremstyle{plain}
\newtheorem{theorem}{\bf Theorem}[section]

\newtheorem{proposition}[theorem]{\bf Proposition}

\newtheorem{conjecture}[theorem]{\bf Conjecture}

  \def\cC{{\mathcal C}}       
\def\d{\delta}  \def\cE{{\mathcal E}}       
\def\D{\Delta}         
   \def\cH{{\mathcal H}}       
\def\e{\eta}           
\def\l{\lambda} \def\cM{{\mathcal M}}       
  \def\cX{{\mathcal X}}       

\def\Z{{\mathbb Z}}    
\def\Q{{\mathbb Q}} 

\newcommand{\mg}{\mathcal{M}_g}

\newcommand{\Mg}{\overline{\mathcal{M}}_g}
\newcommand{\stMg}{\overline{M}_g}

\newcommand{\mgn}{\mathcal{M}_{g,n}}

\newcommand{\Mgn}{\overline{\mathcal{M}}_{g,n}}
\newcommand{\stMgn}{\overline{M}_{g,n}}

\newcommand{\stMgnm}{\overline{M}_{g,n-1}}

\newcommand{\mrat}{\mathcal{M}_{0,2g+2}/S_{2g+2}}
\newcommand{\Mrat}{\widetilde{\mathcal{M}}_{0,2g+2}}
\newcommand{\M}{\overline{\mathcal{M}}}

\newcommand{\Mgg}{\overline{\mathcal{M}}_{g,g}}

\newcommand{\hg}{\mathcal{H}_g}

\newcommand{\Hg}{\overline{\mathcal{H}}_g}
\newcommand{\stHg}{\overline{H}_g}

\newcommand{\hgn}{\mathcal{H}_{g,n}}

\newcommand{\Hgn}{\overline{\mathcal{H}}_{g,n}}
\newcommand{\stHgn}{\overline{H}_{g,n}}

\newcommand{\stHgnm}{\overline{H}_{g,n-1}}

\newcommand{\dmg}{\partial\mathcal{M}_{g}}

\newcommand{\dhg}{\partial\mathcal{H}_{g}}

\newcommand{\dirr}{\d_{\mathrm{irr}}} 
\newcommand{\Dirr}{\D_{\mathrm{irr}}} 
\def\i{\iota}

\def\Pic{\mathop{\mathrm{Pic }}}
\def\Cl{\mathop{\mathrm{Cl }}}

\def\qq{\quad}

\let\geq\geqslant
\let\leq\leqslant
\def\eq{\begin{equation}}
\def\qe{\end{equation}}

\begin{document}

\title[The moduli space of hyperelliptic  curves with  marked points]{On the Kodaira dimension of the moduli space of hyperelliptic  curves with  marked points}

\author{Irene Schwarz}
\address{Humboldt Universit\"at Berlin, Institut f\"ur Mathematik, 
Rudower Chausee 25, 12489 Berlin, Germany}
\email{schwarzi@math.hu-berlin.de}
\begin{abstract}
It is known  that the moduli space $\Hgn$ of genus $g$ stable hyperelliptic curves with $n$ marked points is uniruled for   $n \leq 4g+5$.  In this paper we consider the complementary case. We calculate the canonical divisor of $\Hgn$ and show that it is effective for $n=4g+6$ and big for $n\leq 4g+7$. This leads us to conjecture that $\Hgn$ has non-negative Kodaira dimension for $n = 4g+6$ and is of general type for $n \geq 4g+7$. 
\ \\
A shortened version of this paper appears as an appendix in \cite{bm}, where this conjecture is proven.
\end{abstract}

\thanks{This paper is part of my PhD Thesis written under the supervision of Prof. G. Farkas at Humboldt Universit\"at Berlin, Institut f\"ur Mathematik.  I wish to thank my advisor for suggesting this topic and for his efficient support, in particular for very helpful discussions on the relation between stack and coarse moduli space.
I especially acknowledge fruitful discussions with Ignacio Barros. From him I learned  that the stack  $\stHgn$  is actually not smooth, with reference to the literature. This killed my original approach to calculate the  the Kodaira dimension of $\Hgn$ by controlling its singularities.  But it has also lead to this corrected version of my work.}


\maketitle

\setcounter{page}{1}

\section{Introduction} \label{section introduction}
 The birational geometry of the moduli spaces $\mg$ and $\mgn$ of genus $g$ curves 
and of genus $g$ curves with $n$ marked points has been studied for a long time, together with their Deligne-Mumford compactifications $\Mg$ and $\Mgn$.  A prominent question is when these spaces are unirational, i.e. explicitly describable (at least generically) by finitely many complex parameters, or on the contrary of general type, i.e. of maximal Kodaira dimension. Such investigations go back at least to the fundamental papers \cite{hm} (for $\Mg$) and \cite{l} (for $\Mgn$), establishing that these spaces are of general type if $g$ is sufficiently large or $n$ is sufficiently large, depending on $g$. Subsequently, the results of these papers have been refined by various authors, see e.g.\cite{f,f2,fv4}.
On the other side, there is a finite number of values of $g,n$ for which the moduli space $\Mgn$ is uniruled or even unirational.
 
It is of great geometric interest to understand in a similar way the birational geometry of subvarieties of $\Mg$ and $\Mgn$. Here, the most classical example that comes to mind is, probably, the locus $\hg$ (and $\hgn$) in $\mg$ (or $\mgn$) of hyperelliptic genus $g$ curves (with $n$ marked points). We refer to \cite{acgh, acg} for background on hyperelliptic curves and their moduli space $\hg$ (as well as the associated moduli stack).
 
 
 Clearly, $\hg$ always is unirational, being explicitly parametrized by 
equations $y^2= f(x)$, where $f$ is a polynomial of degree $2g+1$ or $2g+2$ with simple zeroes. By adding marked points, this becomes less explicit. We recall, however,  that it is proved in  \cite{be}  that $\Hgn$ is uniruled for all $n \leq 4g+4$, by applying the methods of that paper  also to the subvarieties $\Hgn$ of $\Mgn$. This has been extended to the case $4g+5$ in \cite{ab}.  In the present paper we shall study the complementary case. Our main result is
 
\begin{theorem}    \label{t1}
The canonical divisor of the moduli space $\Hgn$ is effective for $n=4g+6$ and big for $n\geq 4g+7$.
\end{theorem}

We recall that for $\Mgn$ it was proven in \cite{hm}, \cite{l} that the singularities of the space $\Mgn$ do not impose adjunction conditions. Thus the Kodaira-Iitaka dimension of the canonical divisor is well defined and equal to the Kodaira dimension of the space $\Mgn$.  If one could prove an analogous result for $\Hgn$,  Theorem \ref{t1} would give a result on its  Kodaira dimension. Even without controlling the singularities, it seems reasonable to conjecture

\begin{conjecture} \label{conjecture}
The moduli space $\Hgn$ is of general type for $n\geq 4g+7$ and has non-negative Kodaira dimension for $n=4g+6$.
\end{conjecture}
 
In fact, in an earlier version of this paper we have attempted to control the singularities of $\Hgn$ under the assumption that the corresponding stack $\stHgn$ is smooth. Unfortunately, it turns out that $\stHgn$ is, in fact, not smooth and $\Hgn$ no longer has finite-quotient singularities (see \cite{bm}). 
Therefore we can not use the Raid-Tai criterion to control the singularities and can not identify the Kodaira dimension of $\Hgn$ with the Kodaira-Iitaka dimension of its canonical divisor. 

Nonetheless, using a different approach it was proven by Barros and Mullane in \cite{bm} that Conjecture \ref{conjecture} is actually true and that case $n=4g+6$ does indeed give us an intermediary Kodaira dimension. A shortened version of the present paper appears as an appendix to \cite{bm}.

 \ \\
The main steps in our proof are to calculate the class of the canonical divisor in the rational class group $\rm{CH}^1\left(\overline{\mathcal{H}}_{g,n}\right)\otimes\Q$. Using divisors from \cite{l} and the positivity of the sum of $\psi$-classes we decompose this divisor class into the sum of effective and big divisors.
In Section \ref{section Hgn} we shall prove

\begin{theorem}\label{canonical Hgn}
The canonical class of the stack $\stHgn$ is
\begin{equation}\label{canonical Hgn eq stack}
\begin{split}
K_{\stHgn} &= \sum_{i=1}^n \psi_i -(\frac{1}{2}+\frac{1}{2g+1})\e_0+ 
\sum_S \sum_{i=1}^{\lfloor \frac{g-1}{2}\rfloor} (\frac{2(2i+2)(g-i)}{2g+1}-2)\e_{i,S} \\ &+\sum_S\sum_{i=1}^{\lfloor \frac{g}{2}\rfloor} (\frac{2(2i+1)(2g-2i+1)}{2g+1}-3) \d_{i,S}
-2\sum_{|S|\geq 2} \d_{0,S},
\end{split}
\end{equation}
and for $n\geq 2$ the canonical class of the coarse moduli space $\Hgn$ is
\begin{equation}\label{canonical Hgn eq coarse}
\begin{split}
K_{\Hgn} &= \sum_{i=1}^n \psi_i -(\frac{1}{2}+\frac{1}{2g+1})\e_0+ 
\sum_S \sum_{i=1}^{\lfloor \frac{g-1}{2}\rfloor} (\frac{2(2i+2)(g-i)}{2g+1}-2)\e_{i,S} \\ &+\sum_S\sum_{i=1}^{\lfloor \frac{g}{2}\rfloor} (\frac{2(2i+1)(2g-2i+1)}{2g+1}-3) \d_{i,S}
-2\sum_{|S|\geq 2} \d_{0,S}-\sum_{i=1}^g \d_{i,\emptyset} ,
\end{split}
\end{equation}
where the sum is taken over all subsets $S\subset \{1,\ldots,n\}$
and we use $\d_{i;\emptyset}=\d_{g-i, \{1,\ldots,n\}}$.
\end{theorem}
 
Here $\psi_i$ denote the point bundles (or tautological classes) on $\Hgn$ and $\e_0, \e_{i,S}, \d_{i,S}, \d_{0,S}$ are the boundary divisors. All these divisors are introduced in Section \ref{section preliminaries} or \ref{section Hgn}.

 
 The outline of the paper is as follows. In Section \ref{section preliminaries} we collect notation and preliminaries on $\Hg$ and $\Mgn$. In Section \ref{section Hgn} we calculate the canonical divisor of $\Hgn$. In Section \ref{section eff} we introduce effective divisors on $\Hgn$, following \cite{l}. Finally, in Section \ref{section proof} we prove \ref{t1} by decomposing the canonical divisor $K_{\Hgn}$.
 
\section{Preliminaries}\label{section preliminaries}

In this section we want to recall some well known facts about the hyperelliptic locus $\hg \subset \mg$, its compactification $\Hg\subset \Mg$ and its rational divisor class group. We will use the isomorphism of coarse moduli spaces $\Hg\simeq \Mrat:=\M_{0,2g+2}/S_{2g+2}$ (see \cite{al}) to compute its canonical divisor.

We begin by recalling some basic facts about hyperelliptic curves.
A hyperelliptic curve of genus $g$ is a smooth curve of genus $g$ admitting a degree 2 morphism to $\mathbb{P}^1$ which by the Hurwitz formula will be ramified in exactly $2g+2$ points. The map induces an involution called the hyperelliptic involution. The $2g+2$ ramification points, i.e. the fixed points of the hyperelliptic involution, are called Weierstra\ss \ points. A stable hyperelliptic curve is a stable curve admitting a degree 2 morphism to a stable rational curve. The induced involution is also called hyperelliptic involution. In both cases the degree 2 morphism
is unique and the (stable) hyperelliptic curve can be recovered from its $2g+2$ branch points.

We define $\hg\subset \mg$ as the locus of all (classes of) smooth hyperelliptic curves of genus $g$. We define $\Hg$ as the closure of $\hg$ in $\Mg$ which turns out to be the locus of stable hyperelliptic curves. 

Before we study the boundary of $\Hg$ we will recall the boundary and tautological classes on $\Mgn$.
For results on $\Mgn$ we refer to the book \cite{acg}. 
We emphasize that \cite{acg} mainly works on the moduli stacks. However, all the basic divisors which we shall soon introduce exist both on the stack and its associated coarse moduli space. When it becomes necessary we shall always indicate in notation where we are working.
 All divisor class groups are taken with rational coefficients and, in particular, we identify the the divisor class group on the moduli stack with that of the corresponding coarse moduli space.
We caution the reader that by a standard abuse of notation we will consistently use the same symbol for classes on different moduli spaces. \\

In order to describe the relevant boundary divisors on $\Mgn$, we recall that $\Delta_0$ (sometimes also called $\Delta_{\mathrm{irr}}$) on 
$\Mg$ is the boundary component consisting of all (classes of) stable curves of arithmetical genus $g$, having at least one nodal point with the property that a partial normalisation of the curve at this node preserves connectedness. Furthermore, $\Delta_i$,  for $1 \leq i \leq \lfloor\frac{g}{2}\rfloor,$ denotes the boundary component of curves possessing a node of order $i$ or of type $\d_i$ (i.e. a partial normalisation at this node decomposes the curve in two connected components of arithmetical genus $i$ and $g-i$ respectively). Similarly, on $\Mgn$, we denote by $\Dirr$ the pull-back of $\D_0$ and, for any subset $S \subset \{1, \ldots,n\}$, we denote by $\Delta_{i,S}, 0 \leq i \leq \lfloor\frac{g}{2}\rfloor,$ the boundary component consisting of curves possessing a node of order $i$ such that, after taking a partial normalisation, the connected component of genus $i$ contains precisely the marked points labelled by $S$. Note that, if $S$ contains at most 1 point, one has $\Delta_{0,S}= \emptyset$ (the existence of infinitely many automorphisms on the projective line technically violates stability). Thus, in that case, we shall henceforth consider $\Delta_{0,S}$ as the zero divisor.

We shall denote by $\delta_i, \delta_{i,S}, \dirr$ the rational divisor classes of $\Delta_i, \Delta_{i,S}, \Dirr$ in $\Pic \mg$ and $\Pic \Mgn$, respectively. Note that $\delta_0$ is also called $\delta_{\mathrm{irr}}$ in the literature, but we shall reserve the notation $\delta_{\mathrm{irr}}$ for the pull-back of $\d_0$ under the forgetful map
$\pi:\Mgn \to \Mg$.

We write $\d$ for the sum of all boundary divisors and set $\d_{i,s}=\sum_{|S|=s}\d_{i,S}$. \\

Finally we recall the notion of the point bundles $\psi_i, 1 \leq i \leq n,$ on $\Mgn$. Informally, the line bundle $\psi_i$ (sometimes called the cotangent class corresponding to the label $i$) is given by choosing as fibre of $\psi_i$ over a point $[C;x_1, \ldots, x_n]$ of $ \mgn$ the cotangent line $T_{x_i}^v(C)$. \\


Now let us return to the moduli space $\Hg$ and recall some well known facts (see \cite[Chapter XIII §8]{acg}). 

The locus $\hg\subset \mg$ is a subspace of dimension $2g-1$. It is irreducible and closed in $\mg$. Its closure in $\Mg$ is the locus of stable hyperelliptic curves $\Hg$. The corresponding stack $\stHg$ is smooth. Therefore, we can identify the rational Picard groups of $\Hg$ and $\stHg$. The boundary is $\dhg:= \Hg\setminus \hg =\Hg\cap \dmg$ and we can look at the intersection of $\Hg$ with each (irreducible) component $\D_i$ of $\dmg$ independently. The components of the boundary $\dhg$ are the components of these intersections. 

For $i\geq 2$ a general curve in $\Hg\cap \D_i$ is obtained from smooth hyperelliptic curves $C_1$ and $C_2$ of genera $i$ and $g-i$ by
identifying a Weierstra\ss \ point on $C_1$ with a Weierstra\ss \ point on $C_2$.
When $i = 1$, we must take as $C_1$ a curve in $\M_{1,1}$ and attach it to $C_2$ at the marked point. When in addition $g=2$, $C_2$ must also be a curve in $\M_{1,1}$, attached to $C_1$ at the marked point. By the usual abuse of notation we denote $\Hg\cap \D_i$ as $\D_i$ and its class as $\d_i$.

The case $\Delta_{0}\cap \Hg$ is more complicated, because there are different types of nonseperating nodes: a node of type $\e_0$ is a self intersection of a single irreducible component, while a pair of nodes is called of type $\e_i$ if a partial normalisation at either of the nodes remains connected, but taking a partial normalisation at both nodes will decompose the curve into two connected components of genera $i$ and $g-i-1$. Each of these types of nodes corresponds to a different component of $\dhg$. We will denote the boundary component of curves with a node of type $\e_i$ by $\cE_i$ and its class by $\e_i$.

For $g>2$, a general curve in $\cE_0$ is obtained from a smooth
hyperelliptic curve $C$ of genus $g-1$ by identifying two points which are
conjugate under the hyperelliptic involution. For $g=2$ we must instead take $C\in \M_{1,2}$ and identify the two marked points.
A general curve in $\cE_i$ with $i>0$ is obtained from a smooth hyperelliptic curve $C_1$ of genus $i$, a smooth hyperelliptic curve $C_2$ of genus $g-i-1$, a pair $(p_1, q_1)$ of points on $C_1$, conjugate under the hyperelliptic involution of $C_1$, and a pair $(p_2, q_2)$ of points on $C_2$, conjugate under the hyperelliptic involution of $C_2$, by identifying $p_1$ with $p_2$ and $q_1$ with $q_2$. We leave the case $i=1$ or $g-i-1=1$ to the reader.

These are the irreducible components of $\dhg$. The moduli space $\Hg$ intersects each of the Divisors $\D_i\subset \Mg $ transversally for $i\geq 1$. The Divisor $\D_0\subset \Mg$ intersects $\Hg$ transversally in $\cE_0$, but with multiplicity  2 in $\cE_i$ for $i>0$. This is due to the fact that nodes of type $\e_i$ come in pairs. We therefore get the decomposition on $\Hg$
\begin{equation}\label{dirr Hg}
\dirr= \e_0+ 2\sum_{i\geq 1} \e_i
\end{equation}

With these preparations,
we recall \cite[Chapter XIII, Theorem 8.4]{acg}, :
\begin{theorem}
The rational Picard group $\mbox{Pic}(\Hg)\otimes \Q$ is freely generated by the classes $\d_i$ and $\e_i$. For the Hodge class $\l$ we have the relation
\begin{equation} \label{l Hg}
(8g+4)\l= g\e_0 +2\sum_{i=1}^{\lfloor\frac{g-1}{2}\rfloor} (i+1)(g-i)\e_i +4\sum_{i=1}^{\lfloor\frac{g}{2}\rfloor} i(g-i)\d_i.
\end{equation}
\end{theorem}

We recall that in \cite{acg} equation \eqref{l Hg} is proved on the level of stacks, but it is also valid on the level of coarse moduli spaces where we shall use it. In contradistinction, to calculate the canonical divisors, we shall carefully distinguish 
between stack and coarse moduli space.


As stated above, a smooth hyperelliptic curve $C$ admits a unique double cover $C\to\mathbb{P}^1$, the quotient by the hyperelliptic involution, with $2g+2$ simple branch points. In fact we can construct $C$ from these branch points. In other words, there is a canonical isomorphism between $\hg$, the moduli space of smooth hyperelliptic curves of genus $g$, and $\mrat$, the moduli space of rational $(2g+2)$-pointed curves modulo the symmetric group $S_{2g+2}$. We call $\mrat$ the moduli space of smooth $(2g+2)$-marked curves and denote its compactification by $\Mrat:=\M_{0,2g+2}/S_{2g+2}$. This isomorphism can be extended to an isomorphism $\Hg\simeq \Mrat$ (see \cite[Corollary 2.5]{al}). We will use this isomorphism to study the Picard group of $\Hg$ and calculate its canonical divisor.

Let us look at the boundaries of both moduli spaces: the boundary class $\dhg$ consist of the $g$ irreducible components $\e_i$ for $i=0,\ldots, \lfloor \frac{g-1}{2}\rfloor$ and $\d_i$ for $i=1, \ldots, \lfloor\frac{g}{2}\rfloor$. On $\M_{0,2g+2}$ two boundary components $\d_{0,S}$ and $\d_{0,T}$ will be identified by the action of the symmetric group $S_{2g+2}$ if and only if $S$ and $T$ have the same cardinality. Therefore (by the usual abuse of notation) we denote the boundary components on $\Mrat$ - corresponding to boundary divisors of $\Hg$ -  as $\D_{0,s}$ and their classes by $\d_{0,s}$ where $ s=2, \ldots, g+1$.

\begin{proposition}\label{isomorphism}
Under the canonical isomorphism $\phi: \Hg  \to  \Mrat$ the boundary components $\d_i$ on $\Hg$ will correspond to $\d_{0,2i+1}$ on $\Mrat$ and $\e_i$ will correspond to $\d_{0,2i+2}$, for all $i$, more precisely
\begin{equation}  \label{iso}
\phi^*(\d_0)=\dfrac{1}{2}\e_0, \qquad \phi^*(\d_{0,2i+2})= \e_i, \qquad  \phi^*(\d_{0,2i+1})= 2\d_i
\end{equation}
\end{proposition}
\begin{proof}
 For a general curve in each $\D_{0,s}\subset \Mrat$ we will construct a double cover, simply ramified in the marked points, possibly ramified in the nodes and unramified everywhere else.
 
 Let us begin with a general curve $C$ in $\D_{0,2i+1}$. The curve $C$ consists of a general $2i+1$-marked rational curve $C_1$ intersecting a general $2(g-i)+1$-marked curve rational $C_2$ in a general point. A double cover of a smooth rational curve must always be ramified in an even number of points. Therefore the cover of $C$ must consist of a double cover of $C_1$, ramified in the $2i+1$ marked points and the node, and a double cover of $C_2$ ramified in the $2(g-i)+1$ marked points and the node. This means that the preimage of $C$ under the double cover must consist of general hyperelliptic (or possibly elliptic) curves of genera $i$ and $g-i$ intersecting in Weierstra\ss \ points. Clearly this is an element of $\D_i$. 

Likewise, a general curve $C$ of $\D_{0,2i+2}$ consists of a general $2i+2$-marked rational curve $C_1$ intersecting a general $2(g-i)$-marked rational curve $C_2$ in a general point. A double cover of $C$ consists of  double covers of $C_1$ and $C_2$ ramified in the marked points, but not in the node, which will be general hyperelliptic (or possibly elliptic) curves of genera $i$ and $g-i-1$. These curves will intersect twice in the two conjugate points lying above the node of $C$. For $i>0$, this shows the correspondence with $\cE_i$. 
For $i=0$, the curve $C_1$ is rational and $C_2$ has genus $g-1$. We have to remember that a rational component meeting the rest of the curve in exactly two nodes violates stability and must be contracted. This causes a self intersection on the irreducible component of genus $g-1$. Thus $\D_{0,2}$ corresponds to $\cE_0$.

Now, take one-parameter family $F$ in $\Mrat$ intersecting the boundary divisor $\D_i$ once transversally in a general point. Then we have:
\begin{equation}
1=\mbox{deg}_F \ \d_i= \mbox{deg}_{\phi^*F}\ \phi^*(\d_i)
\end{equation}
From \cite[Ch. XIII, eq. 8.7]{acg} we recall that
\begin{equation}
\mbox{deg}_{\phi^*F}\ \phi^*(\d_0)=2\e_0, \qquad
\mbox{deg}_{\phi^*F}\ \phi^*(\d_{2i+2})=\e_i, \qquad
\mbox{deg}_{\phi^*F}\ \phi^*(\d_{2i+1})=\dfrac{1}{2} \d_i
\end{equation}
Note that \cite{acg} proves these equations on the stacks, but they also hold on the moduli spaces.
Comparing the two equations finishes the proof.
\end{proof}

We can now calculate the canonical divisor of $\Hg$ and $\stHg$. 
 \begin{theorem} \label{canonical Hg}
The canonical divisor of the coarse moduli space $\Hg$ is 
\begin{equation}  \label{candivcoarse}
K_{\Hg}=-(\frac{1}{2}+\frac{1}{2g+1})\e_0+ \sum_{i=1}^{\lfloor \frac{g-1}{2}\rfloor} 
(\frac{2(2i+2)(g-i)}{2g+1}-2)\e_i + \sum_{i=1}^{\lfloor \frac{g}{2}\rfloor} 
(\frac{2(2i+1)(2g-2i+1)}{2g+1}-4)\d_i,
\end{equation}
while 
the canonical divisor of the stack $\stHg$ is given by
\begin{equation}  \label{candivstack}
K_{\stHg}=-(\frac{1}{2}+\frac{1}{2g+1})\e_0+ \sum_{i=1}^{\lfloor \frac{g-1}{2}\rfloor} 
(\frac{2(2i+2)(g-i)}{2g+1}-2)\e_i + \sum_{i=1}^{\lfloor \frac{g}{2}\rfloor} 
(\frac{2(2i+1)(2g-2i+1)}{2g+1}-3)\d_i.
\end{equation}
\end{theorem}

\begin{proof}
We start with  the canonical divisor on the coarse moduli space $\Hg$,
using the canonical isomorphism of Proposition \ref{isomorphism}. This isomorphism does not exist on the level of stacks.
We recall the canonical divisor on the coarse moduli space $\Mrat$ from \cite[Lemma 3.5]{km}. (Note that a general element in $\delta_{0,2}$ has automorphism of order two. So in the notation of \cite{km} one has $\delta_{0,2}=\frac{1}{2}\widetilde{B}_2$.)
\begin{equation}  \label{candivrational}
K_{\Mrat}= -(\frac{1}{4}+\frac{1}{4g+2})\d_{0,2}+ \sum_{s=3}^{g+1} (\frac{s(2g+2-s)}{2g+1}-2)\d_{0,s}
\end{equation} 
Thus equation \eqref{candivcoarse} 
follows from Proposition \ref{isomorphism} by pullback. For computing the canonical divisor of the stack, we need the ramification divisor $R$ for the map
$$ \epsilon: \stHg \to \Hg.$$
The divisor $R$ can be read off the appropriate automorphism groups. First note that a generic element of $\Hg$ carries only the hyperelliptic involution as an automorphism. Likewise, a generic element of the boundary divisor $\e_i$ only carries the hyperelliptic involution, while the generic elements of $\d_i$ have automorphism group 
$\Z_2 \times \Z_2$ (the hyperelliptic involution acts independently on both components, since the node is a Weierstra{\ss}   point). It follows that the map $\epsilon$ is simply ramified over the boundary components $\d_i$, giving
$R=\sum \d_i.$ 
Thus equation \eqref{candivstack} follows from
$$ K_{\stHg}= \epsilon^*(K_{\Hg}) + R .$$
\end{proof}

\section{The locus of pointed hyperelliptic curves}\label{section Hgn}

In this section we shall study the moduli space $\Hgn$ of n-pointed stable hyperelliptic curves of genus $g$ and calculate its canonical class. Unlike in the case of $n=0$ the stack $\stHgn$ is not smooth. Therefore, the rational Picard groups of $\Hgn$ and $\stHgn$ may not be isomorphic. Recall, however, that we still have a natural isomorphism on codimension one cycles 
\begin{equation}
\label{appendix:eq1}
{\rm{CH}}^1\left(\overline{\rm{H}}_{g,n}\right)\otimes\Q\qq\cong{\rm{CH}}^1\left(\overline{\mathcal{H}}_{g,n}\right)\otimes\Q\qq.
\end{equation}
In particular, this will allow us to identify the classes, introduced in the rest of the section, on the stack and the coarse moduli space.

\ \\
We define $\hgn$ (and $\Hgn$) as the moduli spaces of (stable) hyperelliptic curves of genus $g$ together with $n$ distinct marked points (in the stable case nodes can not be marked.) Denoting the canonical projection by $\pi: \Mgn \to \Mg$, we get $\Hgn=\pi^{-1}(\Hg)$ and $\hgn=\pi^{-1}(\hg)\cap \mgn$. Both $\hgn$ and $\Hgn$ are irreducible of dimension $2g-1+n$.

The boundary of $\Hgn$ consist of the following irreducible components:
$\cE_{i,S}$ for $0\leq i\leq \lfloor\frac{g-1}{2}\rfloor$ and $ S\subset \{1,\ldots,n\} $, consisting of those curves in $\cE_i$ such that exactly the marked points labelled by $S$ are on the component of genus $i$;
$\D_{i,S}$ for $ 1\leq i\leq \lfloor\frac{g}{2}\rfloor$ and $S\subset \{1,\ldots,n\}$ consisting of curves in $\D_i$ such that exactly the marked points labelled by $S$ are on the component of genus $i$ and 
$\D_{0,S}:=\D_{0,S}\cap \Hgn$ for  $|S|\geq 2$, where by the usual abuse of notation we use $\D_{0,S}$ for both the divisor on $\Hgn$ and on $\Mgn$.

We denote the classes of these divisors by $\e_{i,S}$ and $\d_{i,S}$.
If $\i: \Hgn\to\Mgn$ is the inclusion, we denote the $\psi$-classes as $\psi_i:=\i^* \psi_i$.

It is known (see \cite{s}) that

\begin{theorem}
The rational Weil divisor class group $\Cl( \Hgn)\otimes\Q$ has a basis consisting of the $\psi$-classes and all the boundary classes.
\end{theorem}


We can now calculate the canonical classes of both the coarse moduli space $\Hgn$ and its assotiated stack $\stHgn$, i.e. prove Theorem \ref{canonical Hgn}

\begin{proof}
We begin on the level of stacks and consider the commutative diagram
\begin{center}
\begin{tikzpicture}
  \matrix (m) [matrix of math nodes,row sep=2em,column sep=4em,minimum width=2em]
  {\stHgn &  \stMgn   \\
   \stHgnm & \stMgnm        \\
   \vdots & \vdots  \\
   \stHg  &  \stMg \\         };

  \path[-stealth]
    (m-1-1) edge node [above] {$\i_n$} (m-1-2)
            edge node [right]  {$\hat{\pi}_n$} (m-2-1)
    (m-1-2) edge node [left] {$\pi_n$}(m-2-2)
    (m-2-1) edge node [above] {$\i_{n-1}$} (m-2-2)
            edge node [right]  {$\hat{\pi}_{n-1}$} (m-3-1)
    (m-2-2) edge node [left] {$\pi_{n-1}$}(m-3-2)
    (m-3-1) edge node [right]  {$\hat{\pi}_1$} (m-4-1)
    (m-3-2) edge node [left] {$\pi_1$}(m-4-2)
    (m-4-1) edge node [above] {$\i$} (m-4-2)
    (m-1-1) edge [bend right=45] node [left] {$\hat{\pi}$} (m-4-1)
    (m-1-2) edge [bend left=45] node [right] {$\pi$} (m-4-2);   
\end{tikzpicture}
\end{center}

First note that each of the squares in this diagram is Cartesian. This follows immediately from the fact that $\stHgn =\pi^{-1}( \stHg$) and a simple diagram chase.
Next we show that, for all $n\geq 1$, the forgetful map $\hat{\pi}_n$ is a universal family. Recall that the  universal family of a fine moduli space $\cM$ is a morphism $\cC\to \cM$ such that any family $\cX \to S$ in $\cM$ induces an isomorphism $\cX\simeq S\times_{\cM} \cC$, see e.g. \cite{hmo}. In \cite{acg} an analogous property is introduced for the stack of $\stMgn$ which has the properties of a fine moduli space. Recall further that the universal family of $\stMgnm$ is the forgetful map $\pi_n: \stMgn\to\stMgnm$. Now any family $\cX\to S$ in $\stHgnm$ is in particular a family in $\stMgnm$. Therefore we get
\begin{equation}
\begin{split}
\cX &\simeq S\times_{\stMgnm} \stMgn \simeq S\times_{\stMgnm} \stMgn \times_{\stHgnm} \stHgnm \\
&\simeq S\times_{\stHgnm} \stMgn \times_{\stMgnm} \stHgnm \simeq S\times_{\stHgnm} \stHgn.
\end{split}
\end{equation}

Next recall that in a universal family $\phi: \cC\to \cM$ the canonical divisor is given as
\begin{equation}\label{rel dual}
K_{\cC}=\phi^* K_{\cM} +\omega_\phi,
\end{equation}
where $\omega_\phi$ is the relative dualizing sheaf of $\phi$ (and in our particular case it is the sheaf of relative K\"ahler differentials $\Omega_\phi$.) By \cite[Chapter II, Proposition 8.10]{h} on the relative K\"ahler differentials of a fibre product we can calculate the relative dualizing sheave of the map $\hat{\pi}_n$ in the diagram above as
\begin{equation}\label{fib rel dual}
\omega_{\hat{\pi}_n}=\i_n^* \omega_{\pi_n}.
\end{equation}

In \cite{h} this identity is shown for schemes, and here we use it in its version for stacks.
Using the equations \eqref{fib rel dual} and \eqref{rel dual} for both $\stMgn$ and $\stHgn$ we can show by induction over $n$ that
\begin{equation}
\begin{split}
K_{\stHgn}=\hat{\pi}^* K_{\stHg}+ \i_n^* (K_{\stMgn}-\pi^*K_{\stMg})
		=\hat{\pi}^* K_{\stHg}+ \sum_{i=1}^n \psi_i -2\sum_{|S|\geq 2} \d_{0,S}.
\end{split}
\end{equation}

Equation \eqref{canonical Hgn eq stack} in Theorem \ref{canonical Hgn} now follows from Theorem \ref{canonical Hg} (giving the sum over $i$ on the right hand side of \eqref{canonical Hgn eq stack}) and repeated applications of \cite[Lemma 1.2 and 1.3]{ac} (giving the sum over $S$). 

In order to compute the canonical divisor on the coarse moduli space, we consider
the map
$$ \epsilon: \stHgn \to \Hgn$$
and note that
\begin{equation} \label{coarse}
\epsilon^*K_{\Hgn}= K_{\stHgn} -R ,
\end{equation} 
where $R$ is the ramification divisor of $\epsilon$. In order to compute $R$ (for $n \geq 1$) we consider the locus $\Sigma \subset \Hgn$ of pointed curves with a non-trivial automorphism. Then the codimension 1 components of $\Sigma$ are 
\begin{itemize}
\item $\{ (C,x) \in \overline{\cH}_{g,1}; x \mbox{  a Weierstra{\ss}  point} \}, \quad n=1$
\item  $\d_{i,\emptyset}, \qquad (i=1, \ldots, g  \mbox{  and  } n\geq1, g \geq 2).$
\end{itemize}

In each case a general element has automorphism group $\Z_2$:  In the first case this group is generated by the hyperelliptic involution of the curve $C$ which acts as an automorphism of the pointed curve $(C,x)$. In the second case, for $i>1$, the non-trivial automorphism is the hyperelliptic involution on the component of genus $i$, while for $i=1$ it is the involution with respect to the node on the elliptic tail.

Ignoring the first case (which is irrelevant  for our theorem) we find 
$R= \sum_{i=1}^g \d_{i,\emptyset}.$
Thus equation  \eqref{canonical Hgn eq coarse}  follows from equations \eqref{canonical Hgn eq stack} and \eqref{coarse}, completing the proof of the theorem.
\end{proof}

\section{Effective divisors}\label{section eff}

In this section we construct effective divisors on $\Hgn$ via pullback from $\Mgn$. For that purpose we recall the following standard result.

\begin{proposition}  \label{divisor}
Let $f:X \to Y$ be a morphism of projective schemes, $D \subset Y$ be an effective divisor and assume that $f(X)$ is not contained in $D$. Then $f^*(D)$ is an effective divisor on $X$.
\end{proposition}

We begin by recalling from \cite{l} the effective divisors $\mathfrak{D}(g;a_1,\ldots, a_n)$ on $\mgn$ defined as the set of all pointed curves $(C,x_i,\ldots, x_n)$ carrying a $\mathfrak{g}^1_g$ through the divisor $\sum_{i=1}^n a_ix_i$. In particular $\mathfrak{D}(g;g)$ is the well known Weierstra\ss \ divisor. We shall now show that these divisors impose a condition on the marked points, not on the curves, and therefore define an effective divisor on $\hgn$.

\begin{proposition}
The hyperelliptic locus $\hgn$ is not contained in any $\mathfrak{D}(g;a_1,\ldots, a_n)$.
\end{proposition}
\begin{proof}
Take any pointed hyperelliptic curve $(C,x_1,\ldots, x_n)$ and set $D:=\sum a_i x_i$. Then $C$ will carry a $\mathfrak{g}^1_g$ through $D$ if and only if $h^0(D)\geq 2$. By Riemann-Roch this is equivalent to $h^0(K-D)\geq 1$. In other words, there must be an effective divisor $D'$ of degree $g-2$ such that $D+D'\sim K$. However, any effective canonical divisor on a hyperelliptic curve consists of $(g-1)$ pairs of points conjugate under the hyperelliptic involution. This means that there must be either indices $i\neq j$ with $x_i$ and $x_j$ conjugate or some $i$ with $a_i\geq 2$ and $x_i$ a Weierstra\ss \  point. Thus only curves with such special choice of marked points are contained in $\mathfrak{D}(g;a_1,\ldots, a_n)$ and not all of $\hgn$.
\end{proof}

In the following computations we shall use the shorthand
$$ \psi= \sum_{i=1}^n \psi_i, \qquad  \d_{i,s}= \sum_{|S|=s} \d_{i,S}. $$
We define $W_g$ as the class of the compactification of $\mathfrak{D}(g;1,\ldots, 1)$ on $\Mgg$ and recall from \cite[Theorem 5.4 and Theorem 5.5]{l}  the decomposition
\begin{equation}\label{Wg}
W_g = -\l + \psi -0\cdot\dirr -3\d_{0,2}-\frac{g(g+1)}{2}\d_{0,g} - \mbox{other terms},
\end{equation}
where "other terms" means a linear combination of the other boundary divisors with non-negative coefficients.

For a set $S\subset \{1,\ldots,n\}$ of cardinality $g$ we take $\pi_S: \Mgn \to \Mgg$ as the forgetful map forgetting all points not labelled by $S$. Now we construct an effective divisor $W$ on $\Mgn$ by summing over the pull-backs of $W_g$ along $\pi_S$ (and then rescaling this sum):

\begin{equation}\label{W}
W:=\binom{n-1}{g-1}^{-1} \sum_S \pi_S^* W_g = 
-\frac{n}{g}\l + \psi-0\cdot\dirr -\sum_{s\geq 2}b_{0,s}\d_{0,s}
-\mbox{higher order terms},
\end{equation}
with $b_{0,2}=2+\frac{g-1}{n-1}$, $b_{0,n}=\frac{n(g+1)}{2}$ and $b_{0,s}\geq b_{0,2}$ for all $s>2$. 

To calculate those coefficients note that we are summing over $\binom{n}{g}$ different pull-backs. Each of these contains the class $\l$ with coefficient $-1$ and the class $\d_{0,n}$ with coefficient $-\frac{g(g+1)}{2}$. A single pull-back  $\pi_S^* W_g$ contains the class $\psi_i$ with coefficient $1$, if $i\in S$, and with coefficient zero otherwise. Likewise, the coefficient of $\d_{0,\{i,j\}}$ is $-1$ for $i\in S, j\notin S$ (or the other way around). The coefficient is $-3$ for $i,j \in S$ and zero for $i,j\notin S$. Thus the coefficient of $\d_{0,2}$ in $\sum_S \pi_S^* W_g$ is $-2\binom{n-2}{g-1}-3\binom{n-2}{g-2}=-2\binom{n-1}{g-1}-\binom{n-2}{g-2}$.

By abuse of notation we write $W=\iota^*_n W$, where $\iota_n$ is the inclusion $\iota_n: \Hgn\to \Mgn$.

\section{Proof of Theorem \ref{t1}}\label{section proof}

In this section we study when the canonical divisor $K=K_{\Hgn}$ is effective or big, i.e. the sum of an ample and effective divisors. Note that, since $\Hgn$ is not $\Q$-factorial, the Kodaira-Iitaka dimension of a divisor may not be well-defined. But we can still determine that a divisor is effective or big by writing down explicit decompositions.

The divisor class $\psi=\sum_{i=1}^n \psi_i$ is big on $\Hgn$ because it is big on $\Mgn$. Thus we will try to write $K$ as a sum of a positive multiple of $\psi$ and effective divisors. We will use the divisor $W$ introduced in Section \ref{section eff} and write
\begin{equation}
K=\epsilon \psi +(1-\epsilon)W +E
\end{equation}
with 
\begin{equation}\label{E}
E=e_{0}\eta_{0}+\sum_{i=1}^{\lfloor\frac{g-1}{2}\rfloor}\sum_S e_{i,S}\eta_{i,S}+\sum_{i=0}^{\lfloor\frac{g}{2}\rfloor}\sum_S d_{i,S}\delta_{i,S}.
\end{equation}

We begin by taking the decomposition of the Hodge class from equation \eqref{l Hg} and pulling it back to $\Hgn$. 
\begin{equation}\label{l Hgn}
(8g+4)\l=g\e_0+2\sum_S\sum_{i\geq 1} (i+1)(g-i)\e_i +\sum_S \sum_{i\geq 1} i(g-i)\d_{i,S}.
\end{equation}

Likewise we pull back equation \eqref{dirr Hg} to get
\begin{equation} \label{dirr Hgn} 
\dirr= \e_0 +2\sum_S \sum_{i\geq 1} \e_i.
\end{equation}

Now, we combine \eqref{l Hgn} and \eqref{dirr Hgn} with \eqref{canonical Hgn eq coarse} and \eqref{W} to study the sign of each component $e_{0}, e_{i,S}$ and $d_{i,S}$ from equation \eqref{E} individually.

Clearly, the coefficients of $e_{i,S}$ and $d_{i,S}$ with $i\geq 1$ are all positive.
In fact, a short computation shows that the coefficients in the decomposition of $K$ - given in \eqref{canonical Hgn eq coarse} - are all positive and the coefficients of $W$ - given in \eqref{W} - are all negative. Thus we are subtracting a negative number from a positive one.

The coefficient $d_{0,S}$ with $|S|=s\leq n-1 $ is 
$$-2+(1-\epsilon) \cdot b_{0,s}\geq -2+(1-\epsilon)\cdot b_{0,2}= -2+(1-\epsilon)(2+\frac{g-1}{n-1})$$
which is positive for $\epsilon$ sufficiently small.

We consider separately the coefficient $d_{0,\{1, \ldots, n\}}=\d_{g,\emptyset}$ which is
$$ -3 + (1- \epsilon) b_{0,n} = -3 + (1- \epsilon) \frac{n(g+1)}{2} > 0, $$
for $g \geq 2$ and $n>g$.

The problematic case is $\e_0$ which only appears as part of $\l$ in $W$. Its coefficient is
$$ e_0=-(\dfrac{1}{2}+\dfrac{1}{2g+1})+(1-\epsilon)\frac{n}{g}\cdot \frac{g}{8g+4} $$
which vanishes for $n=4g+6$ and $\epsilon=0$. It is positive for $n\geq 4g+7$ and $\epsilon$ sufficiently small.

Thus $K$ is big for $n\geq 4g+7$ and still effective for $n=4g+6$.

\end{document}